\documentclass[12pt, a4paper]{article}
\usepackage[english]{babel}

\usepackage{amsmath,amsfonts,amssymb,amsthm} 

\usepackage[alphabetic]{amsrefs} 
\usepackage{xypic}

\newtheorem{theorem}{Theorem}

\newtheorem{example}{Example}

\theoremstyle{definition}
\newtheorem{definition}{Definition}
\newtheorem{remark}{Remark}

\newcommand{\g}{\mathfrak{g}}

\title{Adjoint algebraic groups as automorphism groups of a projector on a central simple algebra}

\author{Viktor A. Petrov, Andrei V. Semenov\footnote{The first author was supported by Young Russian Mathematics award and by RFBR grant 18-31-20044. The second author was supported by ``Native Towns'', a social investment program of PJSC ``Gazprom Neft''}}

\begin{document}

\maketitle

\abstract{We show that any adjoint absolutely simple linear algebraic group over a field of characteristic zero is the automorphism group of some projector on a central simple algebra. Projective homogeneous varieties can be described in these terms; in particular, we reproduce quadratic equations by Lichtenstein defining them.}

\bigskip

Orthogonal and symplectic groups can be defined as automorphism groups of a symmetric (or skew-symmetric) form, which can be viewed as a tensor of valency (0,2) on vector representation. Gordeev and Popov proved that any affine algebraic group is the automorphism group of some tensor of valency (1,2) on some representation (see \cite{6}). In this paper we show that in the case of adjoint absolutely simple groups one can choose a tensor of valency (2,2) in any absolutely irreducible representation, which is closely related to the Casimir operator. A similar result was proven by Brown in the case of spinor representation (see \cite{2}). This tensor naturally arises in theory of knot invariants (see \cite{1}) and in the classical invariant theory (see \cite{7}).

From another hand-side Weyl found that any adjoint classical group can be viewed as the automorphism group of a central simple algebra with an involution \cite{9}. Garibaldi shows (see \cite{4}) that adjoint groups of type $E_7$ can be obtained as the automorphism group of some endomorphism on a central simple algebra of degree $56$. One can identify the image of this endomorphism with the corresponding Lie algebra. Moreover, projective homogeneous varieties can be identified with varieties of right ideals of this algebra related with the endomorphism. 

We generalize these results on the case of any absolutely irreducible projective representation. As a consequence we obtain a new proof of Lichtenstein theorem (see \cite{8}) on equations defining projective homogeneous varieties. 

Let $\g$ be an absolutely simple Lie algebra over a field $K$ of characteristic $0$, $G^{ad} = Aut(\g)^\circ$ be the corresponding adjoint group, and let $(-,-)$ be its Killing form inducing a natural isomorphism $\kappa : \g \longrightarrow \g^*$. Fix an absolutely irreducible representation $\rho : \g \longrightarrow End(V)$. If $\alpha: V \otimes V^{*} \longrightarrow End(V)$ stands for a canonical isomorphism, then the conjugate map $\alpha^*$ gives us an isomorphism between $End(V)^*$ and $V \otimes V^*$.

\begin{definition} Define the map $\pi:End(V) \to End(V)$ as the composition
$$
\xymatrix{
End(V)\ar^-{(\alpha\alpha^*)^{-1}}[r]& End(V)^*\ar[r]^-{\rho^*}&\g^*\ar[r]^-{\kappa^{-1}}&\g\ar[r]^-{\rho}& End(V).
}
$$
\end{definition}

\begin{definition} The image of $\pi$ in $End(V\otimes V)$ by the isomorphism 
$$
End(V)^*\otimes End(V) \simeq End(V)\otimes End(V)\simeq End(V\otimes V)
$$
will be denoted $t$.
\end{definition}

\begin{theorem}
\begin{enumerate}
    \item $\pi$ is a $\g$-equvariant map.
    \item $Im$ $\pi$ = $\rho(\g)$, and, moreover, there exist a constant $c\in K^\times$ such that $\pi^2 =c\cdot \pi$.
    \item If $Aut(\pi)$ is the set of all $f\in PGL(V)$ such that $\pi(f\varphi f^{-1})=f\pi(\varphi)f^{-1}$ for all $\varphi\in End(V)$, then $G^{ad}=Aut(\pi)^\circ$.
    \item $t$ is equal to the image of the Casimir operator of $\g$. 
\end{enumerate}
\end{theorem} 
\begin{proof}
The first statement follows from the fact that all maps in the definition of $\pi$ are $\g$-equivariant. Moreover, 
$$
\pi^2=\rho\kappa^{-1}\rho^*(\alpha\alpha^*)^{-1}\rho\kappa^{-1}\rho^*(\alpha\alpha^*)^{-1},
$$
and a factor $\kappa^{-1}\rho^*(\alpha\alpha^*)^{-1}\rho$ is a $\g$-equivariant map from $\g$ to $\g$, so there exists $c \in K \setminus \{0\}$ such that $\kappa^{-1}\rho^*(\alpha\alpha^*)^{-1}\rho = c \cdot id_{\g}$, hence the second statement holds. \par
In order to prove item 3 note that for any $f\in Aut(\pi)$ the automorphism $Int_f$ preserves $Im$ $\pi$ = $\rho(\g)$, so one can define a homomorphism
$$
Aut(\pi)\to Aut(\g).
$$
If some $f$ lies in the kernel of this homomorpism then it commutes with $\rho(\g)$. But $V$ is absolutely irreducible, so $f$ must be trivial in $PGL(V)$. Hence $Aut(\pi)\le Aut(\g)$. But $\pi$ is $G$-equivariant, so 
$$
G^{ad}=Aut(\g)^\circ\le Aut(\pi)\le Aut(\g),
$$
and statement 3 holds.

Finally, statement 4 easily follows from the following commutative diagram:
$$
\xymatrix{
K\ar[r]\ar[d]& End(V)\otimes End(V)^*\ar[d]^-{id\otimes\rho^*}\\
\g\otimes\g^*\ar[r]^-{\rho\otimes id}\ar[d]^-{id\otimes\kappa^{-1}}& End(V)\otimes\g^*\ar[d]^-{id\otimes\kappa^{-1}}\\
\g\otimes\g\ar[r]^-{\rho\otimes  id}\ar[d]^-{\rho\otimes\rho}& End(V)\otimes\g\ar[d]^-{id\otimes\rho}\\
End(V)\otimes End(V)\ar[r]^-{id}& End(V)\otimes End(V).
}
$$
\end{proof}

Faulkner defines a map
$D:V \otimes V^* \rightarrow \g$ by the following rule:
$$(D(v \otimes a), x) = a(\rho(x)v),$$
where $v\in V, a\in V^*$ and $x \in \g$.
It is easy to see that $\pi=\rho D\alpha^{-1}$.

\begin{definition}
By an \emph{inner ideal} we mean a subspace $M$ in $V$, such that $D(m,\varphi)n$ lies in $M$ for any $m,n\in M$ and $\varphi\in V^*$.
\end{definition}

\begin{remark} We can rewrite the definition in terms of $\pi$: $\pi(M\otimes V^*)(M)\le M$.
\end{remark}
Description of all inner ideal of dimension $1$ leads us to Lichtenstein's theorem.

\begin{theorem}[Lichtenstein] Let $V$ be a representation of the highest weight $\lambda$, let $v_0$ be its corresponding weight vector, and let $v$ be a vector from $V$. The following statements are equivalent:
\begin{enumerate}
\item $Kv$ lies in the orbit $Kv_0$ under the action of $G$ on $\mathbb{P}(V)$;
\item $t \left(v\otimes v \right) = (2\lambda+2\delta, 2\lambda) \cdot (v\otimes v)$, where $\delta$ is equal to the half of the sum of all positive roots;
\item $Kv$ is an inner ideal in $V$.
\end{enumerate}
\end{theorem}

\begin{proof}
Statement 1 follows from Statement 2, since $t$ is $G$-equvariant, and Casimir operator acts on irreducible representation $V(2\lambda)\le V(\lambda)\otimes V(\lambda)$ via multiplication by constant $(2\lambda+2\delta, 2\lambda)$.

If $t=\sum_i f_i\otimes g_i$, then $\pi (u\otimes a) = \sum_i a(f_i(u))g_i$. In particular, $\pi(v\otimes a)(v)$ is a convolution of $t(v\otimes v)$ and $a$. So, statement 3 follows from statement 2.

Finally, statement 1 follows from statement 3 by \cite{3}*{Lemma~3.2}.
\end{proof}

\begin{example}
Let $V$ be the adjoint representation of $\g$. The condition that a line $Kv$ belongs to the orbit of the highest weight vector can be written as
$$
\frac{1}{2}[v,[v,u]]=(v,u)v\text{ for all }u\in V.
$$
\end{example}

Indeed, one a calculation shows that
$$
\pi(v\otimes(u,-))=ad_{[v,u]},
$$
so item 3 from Theorem 2 can be written as 
$$
[v,[v,u]]=c(v,u)v,
$$
where $c(v,u)$ is a scalar depending $u$ and $v$. But $c$ is bilinear and $G$-equivariant, so it must be a scalar multiple of the Killing form. One can find the coefficient substituting $v=e_{\alpha_{max}}$, $u=e_{-\alpha_{max}}$, where $\alpha_{max}$ stands for the maximal root.

Now let us consider the case of non-trivial Tits algebras.

\begin{theorem}
Let $G^{ad}$ be an adjoint absolute simple algebraic group of inner type, let $A$ be a central simple algebra over $K$, and let $\rho:G^{ad}\to PGL_1(A)$ be an absolute irreducible projective representation. One can find a projector $\pi:A\to A$ (as a vector space), such that $G^{ad}=Aut(\pi)^\circ$, where $Aut(\pi) = \{ f\in PGL_1(A) \mid \pi(f\varphi f^{-1})=f\pi(\varphi)f^{-1} \text{ for all } \varphi\in A\}$. Moreover, the image of $\pi$ is a Lie subalgebra in $A$ isomorphic to $Lie(G^{ad})$.
\end{theorem}
\begin{proof}
Let $G_0^{ad}$ be the split form of $G^{ad}$. Then $G^{ad}$ can be obtained by Galois descent via a cocycle $\xi \in Z^1(K,G_0^{ad})$, and $A$ can be obtained by Galois descent via the image of $\xi$ in $Z^1(K,PGL(V))$. Now, since all constructions are $G_0^{ad}$-equivariant, the statement follows from Theorem~1.
\end{proof}

A subspace $M$ in $V$ corresponds to $\alpha(M\otimes V^*)$ in $End(V)$. So, in the case of non-trivial Tits algebras, one may call a right ideal $I$ in $A$ such that $\pi(I)I\le I$ an \emph{inner ideal}. Faulkner proved that inner ideals correspond to objects in Tits geometry (see \cite{3}*{Theorem, p.8} and the remark after the theorem); incidence in this setting usually rewrites as a condition on the dimension of the intersection (see \cite{4} and \cite{5} for the cases of $E_6$ and $E_7$).

\end{document}